\documentclass[11pt,a4paper,twoside]{article}
\usepackage[centertags]{amsmath}
\usepackage{color}
\usepackage{amsthm,enumerate}
\usepackage[psamsfonts]{amsfonts,amssymb}

\newtheorem{definition}{Definition}
\newtheorem{theorem}[definition]{Theorem}
\newtheorem{proposition}[definition]{Proposition}
\newtheorem{lemma}[definition]{Lemma}

\newtheorem{rem}[definition]{Remark}
\newenvironment{remark}{\begin{rem}  \rm }{\end{rem}}
\newtheorem{rems}[definition]{Remarks}

\newcommand\unit{\hbox{\rm 1\kern-2.8truept l}}
\def\sra{\rightarrow}
\def\M{{\cal M}}
\def\LL{{\cal L}}
\def\EE{{\cal E}}
\def\e{\varepsilon}
\newcommand{\tr}[1]{{\rm tr }(#1)}   
\newcommand{\scal}[2]{\langle #1,#2 \rangle }

\begin{document}


\title{Logarithmic Sobolev inequalities and exponential entropy decay in non-commutative algebras}
\author{
Raffaella Carbone 
\\
\small{Dipartimento di Matematica,
Universit\`a di Pavia,
via Ferrata 1, 27100 Pavia, Italy,}\\
\small{e-mail: raffaella.carbone@unipv.it
}}
\date{}


\date{}
\maketitle

\begin{abstract}
We study the relations between (tight) logarithmic Sobolev inequalities, entropy decay and spectral gap inequalities for Markov evolutions on von Neumann algebras. We prove that log-Sobolev inequalities (in the non-commutative form defined by Olkiewicz and Zegarlinski in \cite{OZ}) imply spectral gap inequalities, with optimal relation between the constants.
Furthermore, we show that a uniform exponential decay of a proper relative entropy is equivalent to a modified version of log-Sobolev inequalities; this entropy decay turns out to be implied by the usual log-Sobolev inequality adding some regularity conditions on the quadratic forms.
\end{abstract}

\smallskip


\section{Introduction}
We consider  a continuous Markov (i.e. identity preserving) semigroup  $P=(P_t)_{t\ge 0}$ of positive operators defined on a von Neumann algebra $\M$, where $\M$ consists of operators acting on a finite dimensional Hilbert space $h$.
This kind of semigroups includes quantum Markov semigroups, but can have milder positivity conditions, and can be seen as a non commutative generalization of semigroups associated with Markov processes with a finite number of states.
In the following, we shall suppose that $P$ has a faithful invariant density $\rho$, with respect to which $P$ is symmetric; this invariant density will induce a family of interpolating $L^p$ norms on $\M$: for $f\in\M$, $\|f\|_p^p=\tr{|\rho^{1/(2p)}f\rho^{1/(2p)}|^p}$, where $\tr{\,}$ denotes the usual trace. This is a standard construction, widely used both in commutative and non-commutative settings (see \cite{Vari,BE,DSC,OZ,Saloff-Coste}).

A classical problem, in the study of markovian evolutions, is the analysis of contractive and asymptotic properties of the associated semigroup with respect to different norms.
These behaviors of the semigroup can be described by variational inequalities involving the quadratic form associated with the infinitesimal generator of the semigroup.
Here we study the relationships between logarithmic-Sobolev inequalities, spectral gap and entropy decay for Markov semigroups in a non commutative framework. 
\\
On this subject, one of the most popular results, for semigroups acting on commutative spaces of functions, is
the equivalence between hypercontractivity and the so called logarithmic Sobolev inequalities, proved by Gross (\cite{Gross}, see also \cite{Federbush}). And a remarkable well known consequence of hypercontractivity is the exponential decay of the relative entropy, with a precise relation between the constants involved in the two conditions. This entropy  decay is really equivalent to hypercontractivity for some diffusions (see \cite{Ba,GZ}), but not in general, and a counterexample can be found in \cite{DPP}. 
In particular, in \cite{BT}, the authors study some modifications of logarithmic Sobolev inequalities, which are weaker than the original ones, but stronger than spectral gap inequalities; one of these inequalities determines the best rate for a uniform exponential decay of the entropy. 
Further, in \cite{BT} one can also find an example (Example $3.9$) where the constant characterizing the best decay of the entropy is not the one coming from logarithmic Sobolev inequalities: it suffices to consider a very simple Markov chain, with two states and non-uniform invariant law. 

We shall precisely define the inequalities we are speaking about in next sections, but by now we simply denote by LSI(c), MLSI(c) and SGI(c) the logarithmic Sobolev,  the modified logarithmic Sobolev and the spectral gap inequality respectively, all with best constant $c$  (Definitions \ref{LSI}, \ref{MLSI}, \ref{SGI}).
We have, for a reversible Markov chain, a quite clear picture, that we summarize in the scheme of Figure \ref{fig:commutative_case}.
In this scheme, we recall the hierarchical order of the different inequalities and the corresponding asymptotic properties of the semigroup.
\\
\begin{figure}[ht]
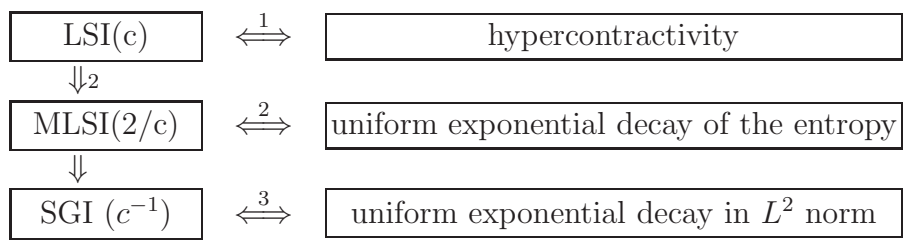
\label{Fig}
	{\large
\begin{center}
\begin{tabular}{ccl}
	\framebox[.2\textwidth]{LSI(c)} &  $\overset{1}\Longleftrightarrow	$ & \framebox[.6\textwidth]{hypercontractivity}\\
	\parbox{1cm}{\noindent$\Downarrow\scriptstyle{2}$} & 	&	\\
\framebox[.2\textwidth]{MLSI(2/c)} &$\overset{2}\Longleftrightarrow	$& \framebox[.6\textwidth]{uniform exponential decay of the entropy}\\
\parbox{1cm}{$\Downarrow$} & 	&	\\
\framebox[.2\textwidth]{SGI $(c^{-1})$} &$\overset{3}\Longleftrightarrow	$ & \framebox[.6\textwidth]{uniform exponential decay in $L^2$ norm}
\end{tabular}
\end{center}}
	\caption{Scheme of hierarchical order in the commutative case}
	\label{fig:commutative_case}
\end{figure}

\noindent
One can read \cite{BT,DSC,Saloff-Coste} to have a picture of what happens in the case of reversible Markov chains with finite state space, and also \cite{Vari,Ba,GZ} for detailed studies on continuous settings. The same references will also give some hints to explore other questions related to log-Sobolev inequalities, such as optimal transport and concentration of measures.

For spaces of non-commutative functions, it is not clear whether the same scheme holds true in general. 
Let us try to detail what was known and what we can add with this note.
The equivalence marked with number 3 was proved for instance in \cite{CF} (see also \cite{C1,C2} for some examples of computation/estimate of the spectral gap for quantum models). 
As for the proof of hypercontractivity for some quantum models, we can say that there exist a variety of results, see \cite{Bia,Boz,C,CaSa,CaLi,Gross2,Krolak,Lindsay,Nelson}, and see also \cite{BZ} for a criterium on the spectrum of the generator which implies hypercontractivity and for other examples. But a general theory which puts hypercontractivity in relation with the log-Sobolev inequality has been lacking for a long time.
A fundamental step was made by Olkiewicz and Zegarlinski (\cite{OZ}), who proved a non commutative generalization of the Gross' theorem (equivalence marked with number 1) under a so-called ``regularity condition'' of the quadratic form associated with the infinitesimal generator of the semigroup; further, they observed that the proper logarithmic Sobolev inequality implies the spectral gap inequality, even if they cannot catch the optimal relation between the constants. A primary merit of their work is also, more generally, that they introduce the appropriate, and not obvious, mathematical objects to treat this kind of problems in non commutative framework.\\
Following the lines of these previous works, here we want to make some additional steps towards the solution of connected open problems, starting from the study of entropy decay, even if, unfortunately, there are some problems which remain to investigate in order to complete the scheme in Figure \ref{Fig}.
\\
- We introduce the suitable non-commutative version of the modified logarithmic Sobolev inequality equivalent to the decay of a relative entropy, naturally related with the norms induced by the invariant density $\rho$ (Theorem \ref{MLS}).
\\
- We show that LSI implies MLSI, at least under some suitable conditions which are a weaker version of the regularity conditions introduced by Olkiewicz and Zegarlinski for the proof of the non commutative Gross' theorem (Proposition \ref{LS-MLS}).
\\
- We prove the optimal relation between spectral gap and log-Sobolev constant (Theorem \ref{LS-SG}), i.e. we show that the spectral gap is not less than the inverse of the log-Sobolev constant (in \cite{OZ} something similar was proved, but with a factor $1/2$). We remark that this proof holds for the case when $h$ is separable too (infinite dimensional).

Furthermore, we observe that, once we have studied the role of all these inequalities, 
when the problem is stated in finite dimensions, it can sometimes be easy to understand whether these inequalities are verified or not, 
 but what is anyway generally not evident, is the evaluation of the best involved constants. Also for classical processes, the exact log-Sobolev constants are rarely known (\cite{ChSh,DSC,Gross}). In the quantum case, this can become even more difficult. Anyway these results already revealed to be useful: see, for instance, \cite{BZ,C,CaSa} for the computation/estimate of quantum log-Sobolev constants and proofs of hypercontractivity for some quantum models.

The organization of the paper is as follows. In Section 2 we concentrate on the decay of the entropy, essentially proving the analogue of the equivalence marked with number 2 in Figure \ref{Fig}: we introduce the proper non commutative entropy for integrable functions and show that its uniform exponential decay is equivalent to a suitable modified log-Sobolev inequality (Theorem \ref{MLS}). In Section 3, we concentrate on the 
hierarchy of different inequalities and we show that the log-Sobolev inequality with constant $c$ implies a spectral gap inequality with constant $c^{-1}$ and a modified log-Sobolev inequality with constant $(2/c)$. The first implication (Theorem \ref{LS-SG}) is proved in general, while, for the latter (Proposition \ref{LS-MLS}), the result needs some regularity condition; this is the reason why we dedicate Subsection 3.1 to the discussion of these aspects. The relations between constants are optimal in both cases. 

{\bf Important note.} Except for this note and the related reference \cite{KT}, this version of the paper dates back to february 2013, even if published here one year later.
In the meanwhile, some new connected results appeared in \cite{KT}. 
Kastoryano and Temme (\cite{KT}) presented, in the same context used here, a very interesting study of log-Sobolev inequalities defined on any $L^p$ space, a clear analysis of the link between different asymptotic and convergence properties of the evolution (see, for instance, \cite[Proposition 13 and Theorem 22]{KT}), and some nice applications. To be more precise on the contact points we have, their paper anticipated part of Theorem \ref{MLS}, the implication ``b)$\Rightarrow $ a)'' (see \cite[Lemma 21]{KT}); while \cite[Theorem 16]{KT} solves, at least under some regularity conditions, one of the ``open questions'' mentioned at the end of this paper.  
Finally, from \cite[Proposition 13 and Theorem 16]{KT} one could deduce the same relation obtained in Theorem 11, even if only in the particular case of regularity assumptions.

\section{Uniform exponential decay of the entropy and a modificaton of logarithmic Sobolev inequality}

We consider a von Neumann algebra ${\cal M}$, acting on a finite dimensional Hilbert space $h$, and a positive, continuous, identity preserving semigroup $P=(P_t)_{t\ge 0}$ with a faithful invariant state $\rho$ (we shall identify the state and its density).
We introduce the $L^p$ spaces associated with $\rho$ in the following way: $\|f\|_p^p=\tr{|\rho^{1/2p}f\rho^{1/2p}|^p}$, for $p\in[1,+\infty)$ and the $\|\;\|_\infty$ norm is the usual norm of the algebra $\M$. $L^2$ is a Hilbert space with respect to the scalar product 
$
\langle f,g\rangle=\tr{\rho^{1/2}f^*\rho^{1/2} g}.
$
\\
We shall suppose that the semigroup $P$ is symmetric with respect to this scalar product (this is the analogue of considering reversible invariant laws for Markov chains).
If we drop this condition, actually very similar results can be proved, but, as for commutative algebras, some details need to be treated differently and the constants will be a little worse.
\\
We denote by $\LL$ the infinitesimal generator of the semigroup and by $\EE$ the associated sesquilinear form with respect to the $L^2$ scalar product
$$
\EE(f,g)=-\langle f,\LL g\rangle=-\tr{\rho^{1/2}f^*\rho^{1/2}\LL g}.
$$
For positive functions $f$ in $\M$, we introduce the relative entropy of $f$ (relative to $\rho$) and we shall denote it by $E(f)$,
$$
E(f)=\tr{\rho^{1/2} f \rho^{1/2}(\lg(\rho^{1/2} f \rho^{1/2})-\lg\rho)}-\| f \|_1 \lg\| f\|_1
$$
Our main result in this section is the characterization of the uniform exponential decay of this relative entropy by means of a proper infinitesimal inequality.
\\
Obviously, there are many interesting and useful possible definitions of entropy, so
one can wonder, first of all, about the opportunity of considering this definition of the entropy. 
\smallskip

{\bf About the choice of the relative entropy}.
A proper definition of the relative entropy, we think, should be chosen considering the context it has to be related with and the classical objects it has to generalize, in the sense that it should coincide with the usual entropy in the commutative case. For the latter aspect, we want to 
recover the lines followed, for instance, in \cite{BT}, \cite{DSC} and \cite{GZ}, but this idea determines only the commutative restriction of the definition. 
However, the choice is essentially ``obliged'' if we want to be coherent with some usual non-commutative definition of the $L^p$ norms associated to a state, which are the ones introduced above (and used, for instance, also in \cite{CF}), and so with the corresponding study of hypercontractivity and log-Sobolev inequalities initiated in \cite{OZ}.
\\
Indeed, also in the classical setting, entropy naturally appears as a result of a proper derivation of some norms; similarly, we can obtain this definition in the non-commutative case. 
In \cite{OZ}, the authors
introduced a family of maps $I_{p,q}$, for $p,q>1$, which are a kind of embedding of the $L^q$ space in the $L^p$ space,
 \begin{equation}\label{Ipq}
 I_{p,q}(f)=\rho^{-1/2p}   (\rho^{1/2q} f \rho^{1/2q})^{q/p}   \rho^{-1/2p},
 \qquad \mbox{ for } f>0,
 \end{equation} 
and an associated operator valued relative entropy $T_q$,
\begin{eqnarray}\label{Tq}
T_q(f)&=&\rho^{-1/2q}(\rho^{1/2q} f \rho^{1/2q})\lg (\rho^{1/2q} f \rho^{1/2q}) \rho^{-1/2q} - \frac{1}{2q}(f\lg \rho+\lg \rho f)
\nonumber\\
&=&-q\frac{d}{ds}I_{q+s,q}(f){\Big{|}}_{s=0}.
\end{eqnarray}
In particular, 
$T_1(f)=f\rho^{1/2}\lg (\rho^{1/2} f \rho^{1/2} )\rho^{-1/2} - \frac{1}{2}(f\lg \rho+\lg \rho f) $
is a selfadjoint operator and we can rewrite the definition of the entropy equivalently as
\begin{equation}\label{Entropy-2}
E(f)=\langle \unit, T_1(f) \rangle-\| f \|_1 \lg\| f\|_1.
\end{equation}
Notice that, when $\rho$ and $f$ commute, $I_{p,q}(f)=f^{q/p}$ and $T_q(f)=f\lg f$, for any $p$ and $q$, so, for commutative spaces of functions, these functions do not depend on the invariant measure, and the operator valued entropy does not even depend on $q$. Moreover the entropy will be $E(f)=\tr{\rho f\lg (f/\| f\|_1)}$, and we recognize the usual classical expression (\cite{DSC}) if we think about $\rho$ and $f$ as diagonal matrices.
\smallskip\\

After these remarks, we go back to the discussion about the exponential entropy decay and state the main result on this point of the subject.

\begin{theorem}\label{MLS}
The following conditions are equivalent:
\\
a) a constant $c$ verifies $E(P_tf)\le e^{-c t}E(f) $ for all $f>0$;
\\
b) a constant $c$ verifies
$ c E(f) \leq \EE(\lg(\rho^{1/2} f \rho^{1/2})-\lg\rho,f)$  for all $f>0$.
\end{theorem}

\begin{definition}\label{MLSI}
We shall call the inequality 
$$ 
c E(f) \leq \EE(\lg(\rho^{1/2} f \rho^{1/2})-\lg\rho,f)
$$ 
a Modified Logarithmic Sobolev inequality of constant $c$ (MLSI$(c)$, for short). 
\end{definition}

The commutative version of this inequality and its relation with other similar inequalities were studied in detail in \cite{BT}, for Markov chains with finite state space (this MLSI is essentially relation (1.5) in \cite{BT}, modulo a factor $1/2$), while it is called a tight (tendue, in french) $L^1$-logarithmic Sobolev inequality in \cite{Ba}.
For diffusive classical Markov processes, i.e. with a ``carr\'e du champ'' operator verifying Liebniz condition, this inequality is equivalent to the log-Sobolev inequality that we shall define later (see always \cite{Ba}), but it is known that this is not true in general in a discrete setting (see \cite{DPP} for a counterexample).

For the proof of the theorem, in this section, we shall use the notation $f_t=P_t f$, for $f$ in $\M$.

\begin{lemma}\label{logDerivative}
For all strictly positive elements $f$,
$$
\frac{d}{dt} \lg (\rho^{1/2} f_t \rho^{1/2})
=\int_0^{+\infty} \frac{1}{s+\rho^{1/2} f_t \rho^{1/2}}\rho^{1/2}\LL f_t \rho^{1/2} \frac{1}{s+\rho^{1/2} f_t \rho^{1/2}}  ds. 
$$
\end{lemma}

\begin{proof}
We use the integral representation of the logarithm for a strictly positive operator $g$, 
$\lg g=\int_0^\infty \frac{1}{s+1}-\frac{1}{s+g} ds$, so
\begin{eqnarray}\label{eqLog}
\frac{d}{dt} \lg g(t)
&=& \lim_{h\sra 0}\frac{1}{h}\int_0^\infty \frac{1}{s+g(t)}-\frac{1}{s+g(t+h)} ds
\nonumber\\
&=&  \lim_{h\sra 0}\int_0^\infty (s+g(t+h))^{-1}\frac{g(t+h)-g(t)}{h}(s+g(t))^{-1} ds
\\
&=&  \int_0^\infty (s+g(t))^{-1}\frac{dg(t)}{dt}(s+g(t))^{-1} ds,\nonumber
\end{eqnarray}
where the last equality holds if we suppose that we have the conditions to interchange limit and integral.
The previous relation is the thesis, if we take $g(t)=\rho^{1/2} f_t \rho^{1/2}$ and consider that $\frac{dg(t)}{dt}=\rho^{1/2} {\cal L}f_t \rho^{1/2}$.
So we simply have to prove that we really can interchange the limit and integral operations in our case.
\\
In this proof, we shall denote by $\lambda(g)$ the minimum eigenvalue of a selfadjoint operator $g$.
Now notice that
$$
(s+\rho^{1/2} f_{t+h} \rho^{1/2})\ge s+\lambda(\rho)\lambda( f_{t+h}).
$$
So, since all these operators are positive, at least for $|h|\le 1$,
$$
\|(s+\rho^{1/2} f_{t+h} \rho^{1/2})^{-1} \|
\le \frac{1}{s+\lambda(\rho)\lambda( f_{t+h})}
\le\frac{1}{s+\lambda(\rho)\varepsilon(t)},
$$
where $\varepsilon(t):=\min_{|h|\le 1}\lambda(f_{t+h})$ (this minimum exists and is strictly positive since the semigroup is continuous and the algebra acts on a finite dimensional space).
\\
Moreover we have, always for $|h|\le 1$, 
$$
\left\|\frac{f_{t+h}-f_{t}}{h}\right\|=\left\|\frac{e^{h{\cal L}}-\unit}{h}(f_t)\right\| 
\le e^{\| \LL\|}\|f_t\|.
$$
So the integrand in \eqref{eqLog}, for the case $g(t)=\rho^{1/2} f_t \rho^{1/2}$ can be controlled in norm,
\begin{eqnarray*}
&&\|(s+\rho^{1/2} f_{t+h} \rho^{1/2})^{-1}\rho^{1/2}\frac{f_{t+h}-f_{t}}{h}\rho^{1/2}(s+\rho^{1/2} f_{t} \rho^{1/2})^{-1}\|
\le\\
&&\kern+7cm\le \frac{e^{\| \LL\|}\|f_t\|}{(s+\lambda\varepsilon(t)) (s+\lambda(f_{t}))},
\end{eqnarray*}
which is integrable on $(0,+\infty)$ as a function of $s$. This concludes the proof since now we are allowed to use Lebesgue's dominated convergence theorem.
\end{proof}

\begin{proof} ({\it of Theorem \ref{MLS}}).
First notice that, by the invariance of $\rho$,
\begin{equation}\label{der}
\frac{d}{dt}\|f_t\|_1=\frac{d}{dt}\tr{\rho P_t f}=\tr{\rho{\cal L}f_t}=0.
\end{equation}
Moreover, $f>0$ (i.e. strictly positive) implies that $P_tf>0$ for all $t$, since, by the positivity of $P$, if we denote by $\lambda(f)$ the minimum eigenvalue of $f$ as before, we have
$$
P_tf \ge \lambda(f) P_t(\unit)
=\lambda(f)\unit>0.
$$
Now, we can compute the derivative of the entropy evolution
\begin{eqnarray*}
\frac{d}{dt} E(f_t) 
&\overset{\mbox{by }\eqref{Entropy-2}}{=}&\frac{d}{dt}\left\{ \langle \unit, T_1(f_t) \rangle - \|f_t\|_1\lg \|f_t\|_1 \right\}
\overset{\mbox{by }\eqref{der}}{=}\frac{d}{dt} \langle \unit, T_1(f_t) \rangle
\\
&&\kern-1cm= \tr{\rho^{1/2} \LL f_t \rho^{1/2}(\lg(\rho^{1/2} f_t \rho^{1/2})-\lg\rho)}\\
&&\kern-1cm +\tr{\rho^{1/2} f_t \rho^{1/2}
\int_0^{+\infty} \frac{1}{x+\rho^{1/2} f_t \rho^{1/2}}\rho^{1/2}\LL f_t \rho^{1/2} \frac{1}{x+\rho^{1/2} f_t \rho^{1/2}}  dx },
\end{eqnarray*}
by Lemma \ref{logDerivative}.
But the term in the last line is null since it coincides with
\begin{eqnarray*}
&&\kern-1cm
\int_0^{+\infty} \tr{\rho^{1/2} f_t \rho^{1/2}\frac{1}{x+\rho^{1/2} f_t \rho^{1/2}}\rho^{1/2}\LL f_t \rho^{1/2} \frac{1}{x+\rho^{1/2} f_t \rho^{1/2}}  } dx 
\\
&=&
\int_0^{+\infty} \tr{\frac{\rho^{1/2} f_t \rho^{1/2}}{(x+\rho^{1/2} f_t \rho^{1/2})^2}\rho^{1/2}\LL f_t \rho^{1/2}  } dx 
\\
&=& \tr{
\int_0^{+\infty} \frac{\rho^{1/2} f_t \rho^{1/2}}{(x+\rho^{1/2} f_t \rho^{1/2})^2}  dx \, \rho^{1/2}\LL f_t \rho^{1/2} } 
\\
&=& \mbox{tr}\left(\left.-\frac{\rho^{1/2} f_t \rho^{1/2}}{x+\rho^{1/2} f_t \rho^{1/2}}\right|_0^\infty \rho^{1/2}\LL f_t \rho^{1/2}\right)
=\tr{\rho \LL f_t }=0.
\end{eqnarray*}
So we conclude
\begin{eqnarray*}
\frac{d}{dt} E(f_t) 
&=& \tr{\rho^{1/2} \LL f_t \rho^{1/2}(\lg(\rho^{1/2} f_t \rho^{1/2})-\lg\rho)}.
\end{eqnarray*}
Now, for any positive real number $c$, we have
\begin{eqnarray*}
\frac{d}{dt} e^{c t} E(f_t) 
&=& e^{c t} (c E(f_t) + \frac{d}{dt} E(f_t))
\\
&=& e^{c t}(c E(f_t) + \tr{\rho^{1/2} \LL f_t \rho^{1/2}(\lg(\rho^{1/2} f_t \rho^{1/2})-\lg\rho)}),
\end{eqnarray*}
so that the previous derivative is non-positive if and only if
$$
c E(f_t) \leq \EE(\lg(\rho^{1/2} f_t \rho^{1/2})-\lg\rho,f_t).
$$
Now the equivalence between the two conditions in the statement of this theorem easily follows:
\\
- if a) holds, then the derivative has to be non-positive at least for $t=0$ and we obtain b);\\
- conversely, if b) holds, then the derivative is never positive and so 
$e^{c t} E(f_t) \le E(f)$ for all $t$.
\end{proof}

\section{Regularity properties of the Dirichlet forms and relations between different variational inequalities}

For commutative reversible Markov processes, as we already outlined, the MLSI introduced in the previous section is implied by a corresponding log-Sobolev inequality and it is equivalent to it for some diffusions. Moreover it implies the spectral gap inequality (\cite{Ba,DSC,GZ}).
\\
Here we want to investigate whether we are in a position to prove similar results for non-commutative algebras.
To do this, first, we have to introduce the mathematical objects involved in these questions, starting with the precise definition of spectral gap and logarithmic Sobolev inequalities that we have mentioned many times.

\begin{definition} \label{SGI}
The spectral gap of the semigroup $P$ (or equivalently of the generator $\LL$) is the best constant $c$ verifying
\begin{equation}\label{SG}
c{\rm Var}(f)\le \EE(f,f)
\end{equation}
for all $f$ in $\M$, where ${\rm Var}(f)=\|f-\tr{\rho f}\|^2$.
We shall call \eqref{SG} a spectral gap inequality of constant $c$ (SGI(c)).
\end{definition}

\begin{definition}\label{LSI}
For $f$ strictly positive in $\M$, we shall denote $H(f):=\langle f,T_2(f)\rangle-\|f\|_2^2\lg\|f\|_2$.
\\
We shall say that the semigroup verifies a logarithmic Sobolev inequality with constant $c$ (LSI$(c)$) if
$$
H(f)\le c\EE(f,f)
$$
for some positive constant $c$ and for all positive $f$ in $\M$.
\end{definition}
When similar objects (SGI and LSI) are defined for non uniformly continuous semigroups, obviously one should ask the argument $f$ to be in the domain of the generator. 

Usually, one can consider a more general form of LSI, which allows the presence of an additional term proportional to 
$\|f\|_2^2$ in the right hand side. But here, for the kind of problems we are going to tackle (relations with the spectral gap and exponential decay of the entropy), the only interesting LSIs are the stronger ones we have introduced in the previous definition (they are sometimes called tight, for instance by Bakry \cite{Ba}).  

We remark that, when $\rho$ and $f$ commute, $H(f)=\tr{\rho f^2\lg (f/\| f\|_2)}=\frac{1}{2}E(f^2)$ and we recover the usual form of LSIs for commutative functions. Notice also that the relation of $H$ with the entropy, in this non-commutative case, is expressed by (recall the definition of $T_2$, in relation \eqref{Tq})
$$
E(f)= 2\langle I_{2,1}(f), T_2(I_{2,1}(f)) \rangle
-2 \| I_{2,1}(f) \|_2^2 \lg\| I_{2,1}(f) \|_2
=2H(I_{2,1}(f)),
$$
which is natural if we remember that, for commutative functions, we have $I_{2,1}(f)=f^{1/2}$ by its definition in relation \eqref{Ipq}.

We want to study the mutual relations of the different kinds of inequalities: SGI, MLSI, LSI. Their relation with the asymptotic properties of the semigroup is now clear also in the non commutative setting, but we resume here the main results for the sake of clarity.
\\
1. The spectral gap inequality is equivalent to a uniform exponential convergence in $L^2$ norm: (see \cite{CF} for the non commutative setting) the spectral gap is the best constant $c$ verifying
$$
\|P_t f-\tr{\rho f}\|_2 \le e^{-ct}\|f-\tr{\rho f}\|_2
\qquad \mbox{ for all } f\in\M.
$$
\\
2. The MLSI is equivalent to the uniform exponential decay of the entropy (Theorem \ref{MLS} here).
\\
3. The non-commutative LSIs were introduced by Olkiewicz and Zegarlinski in \cite{OZ}, and they proved the relation with hypercontractivity:
if $P$ is hypercontractive, then it verifies a LSI; vice versa, if $P$ verifies a LSI and a regularity condition of the quadratic forms, then it is hypercontractive (see \cite[Theorem 3.8]{OZ},  for the precise statement).
\smallskip

As we already underlined in the Introduction, for commutative reversible Markov chains, we have
a precise link between the different inequalities (Figure \ref{fig:commutative_case}). 
The aim of this section is to investigate whether the same is true for our context. We immediately underline that the answer is not complete and we will detail the missing points later. Essentially, we can prove some satisfactory results about the fact that LSIs imply the other two inequalities, but we have not been able by now to conclude, under general conditions, about the relation between MLSI and SGI.
\\
It is important to highlight that it is already known that, if $P$ verifies a LSI with constant $c$, then its spectral gap is not less than $(2c)^{-1}$ (\cite{OZ}, Theorem 4.1). We shall improve the estimate here to the optimal one, while, for proving that LSI implies MLSI, we shall need a regularity condition of the quadratic form. 
In the next subsection we shall introduce and discuss some properties about two different  regularity conditions, but the only necessary element for reading Subsection 3.2 is the notion of weak regularity condition (WRC) in Definition \ref{WRC}.

\subsection{Regularity conditions of the Dirichlet forms}

First of all, we introduce the definition of the two regularity conditions we are going to consider. Then we shall discuss how they are related and show a case when they are always satisfied.
\\
Following \cite{OZ}, and remembering that the maps $I_{p,q}$ are defined in \eqref{Ipq},
 we introduce the quadratic form on $L^q$ as
$$
\EE_q(f,f):=-\scal{I_{p,q}(f)}{\LL f}=-\scal{\rho^{-1/2p}   (\rho^{1/2q} f \rho^{1/2q})^{q/p}   \rho^{-1/2p}}{\LL f},
$$
where $p$ is conjugate to $q$ and $f$ is positive. Notice that $\EE_2(f,f)=\EE(f,f)$. 

\begin{definition}\label{WRC}
We shall say that the semigroup (or the associated quadratic form) verifies the regularity condition (RC) when the following holds for all $q>1$
  $$
\mbox{(RC)}\qquad 
\EE_2(I_{2,q}f,I_{2,q}f)
\le
\frac{q^2}{4(q-1)}\EE_q(f,f),
\quad \mbox{ for all }f>0. $$
We shall say that the semigroup (or the associated quadratic form) verifies the weak regularity condition with constant $\beta$ (WRC-$\beta$) when 
$$
\mbox{(WRC)}\qquad
 \beta \EE(I_{2,1}f,I_{2,1}f)
\le
\EE(\lg(\rho^{1/2} f \rho^{1/2})-\lg\rho,f),
\quad \mbox{ for all }f>0.$$
We shall call the (WRC) standard when the coefficient $\beta$ is equal $4$.\\
\end{definition}

The regularity condition (RC) is a particular case of the one introduced by Olkiewicz and Zegarlinski in \cite{OZ},  Definition 3.6. In their version, an additive term, proportional to $\| f\|^q$, was admitted on the right hand side, but here, because of the tight form of LSIs we consider, only this formulation can be useful.
\\
Both (RC) and standard (WRC) are always verified for reversible Markov processes on commutative spaces; in the non-reversible case, similar conditions hold, but with different constants.
\\
In the following proposition, we prove that standard (WRC) is weaker than (RC).
We show that they are both verified for symmetric trace preserving semigroups in Theorem \ref{trace}, but this result is really new only for what (WRC) is concerned (see Remark \ref{rem-tr}).

\begin{proposition}\label{RC}
(RC) implies standard (WRC).
\end{proposition}
\begin{proof}
For reversible Markov processes on commutative spaces, (WRC) is usually proven directly, by the use of some inequality involving the logarithm, as in \cite{DSC}, for instance.
In this proof, we see how it can also be deduced by (RC) by taking the limit for $q\sra 1$.  
First, we shall rewrite the right-hand side in (RC), for strictly positive $f$, dropping the factor $q^2/4$ which will not give any problem in the limit,
\begin{eqnarray*}
\frac{1}{q-1}\EE_q(f,f)
&=& -\frac{1}{q-1}\tr{\rho^{1/2q}   (\rho^{1/2q} f \rho^{1/2q})^{q-1}   \rho^{1/2q}\LL f}
\\
&&\kern-2.3cm= -\frac{1}{q-1}\tr{(\rho^{1/q}  - \rho)\LL f}
-\frac{1}{q-1}\tr{\rho^{1/2q}  ( (\rho^{1/2q} f \rho^{1/2q})^{q-1}  -\unit) \rho^{1/2q}\LL f}.
\end{eqnarray*}
Now we compute the limits of the two terms separately
\begin{eqnarray*}
\frac{1}{q-1}(\rho^{1/q}-\rho)
&=&\frac{1}{q-1}\rho^{1/q}(\unit-\rho^{(q-1)/q})\\
&=&-\frac{1}{q}\rho^{1/q} \sum_{k\ge 1 }\left(\frac{q-1}{q}\right)^{k-1} \frac{(\lg \rho)^k}{k!}
\;\underset{q\sra 1}{\longrightarrow}\;( -\rho\lg\rho).
\end{eqnarray*}
and similarly
$$
\frac{1}{q-1}((\rho^{1/2q} f \rho^{1/2q})^{q-1} -\unit)=...... \;\underset{q\sra 1}{\longrightarrow}\; \lg(\rho^{1/2} f \rho^{1/2}).
$$
Then
\begin{eqnarray*}
\frac{1}{q-1}\EE_q(f,f)
\;\underset{q\sra 1}{\longrightarrow}&&
\tr{(\rho\lg\rho-\rho^{1/2}  \lg (\rho^{1/2} f \rho^{1/2})   \rho^{1/2})\LL f}\\
&&= \EE(\lg(\rho^{1/2} f \rho^{1/2})-\lg\rho,f).
\end{eqnarray*}
Now it is easy to see that $I_{2,q}(f)\sra_{q\sra 1}I_{2,1}(f)$ and, since the quadratic form $\EE$ is continuous, one can easily conclude
\begin{eqnarray*}
 \EE(I_{2,1}f,I_{2,1}f)
 &=&\lim_{q\sra 1}\EE_2(I_{2,q}f,I_{2,q}f)
\le \lim_{q\sra 1} \frac{q^2}{4(q-1)}\EE_q(f,f)\\
&=&\frac{1}{4} \EE(\lg(\rho^{1/2} f \rho^{1/2})-\lg\rho,f),
\end{eqnarray*}
which is the standard (WRC) property.
\end{proof}

\smallskip

\begin{theorem}\label{trace}
Let $P$ be trace preserving and symmetric. Then the corresponding quadratic forms verify (RC) and standard (WRC).
\end{theorem}

\begin{remark} \label{rem-tr}
(1.) Obviously, when the algebra $\M$ acts on a finite dimensional Hilbert space $h$, the trace can be normalized and we shall obtain an invariant state for the semigroup. The same proof, however, can be easily adapted for the case when $\M$ coincides with the algebra of all bounded operators on a separable Hilbert space $h$.\\
(2.) For the regularity condition (RC), the result is already known (see Theorem 5.5 in \cite{OZ}), but with a different proof. So we could briefly demonstrate the previous statement in the following way: we use \cite{OZ} in order to have (RC) and then deduce (WRC) by Proposition {\ref{RC}}.
However, we think it could be useful to propose the alternative direct proof which follows, since it is not more 
complicated and we hope it can have some interest in that it suggests a way to reduce the problem to the analogous 
result for commutative spaces, while, in {\cite{OZ}}, 
the proof is completely different and uses some properties about reflection positive functions (Section $5$ in \cite{OZ}).  
\end{remark}

\begin{proof}
In the trace case, i.e. when $\rho$ is proportional to the identity, the regularity conditions (RC) and standard (WRC) are written respectively
\begin{eqnarray}\label{reg-trace}
-\tr{f^{q/2}\LL(f^{q/2})}&\le& -\frac{q^2}{4(q-1)}\tr{f^{q-1}\LL(f)},
\\
-4\tr{f^{1/2}\LL(f^{1/2})}&\le& -\tr{(\lg f)\LL(f)}.\nonumber
\end{eqnarray}
\underline{Step 1}. The spectral resolution of $f$. We call $N$ the dimension of the Hilbert space $h$. We consider a positive operator $f$ in $\M$, we call $\sigma(f)$ its spectrum and write its spectral representation 
$$
f=\sum_n f_n|e_n\rangle\langle e_n|=\sum_{\lambda\in\sigma(f)} \lambda Q_\lambda,
$$
where $(f_n)_n$ are the eigenvalues (repeated according to their multiplicity), $(e_n)_n$ is an orthonormal basis of $h$ made of eigenvectors of $f$, and $(Q_\lambda)_\lambda$ are the spectral projections on the eigenspaces of $f$, $Q_\lambda=\sum_{n:f_n=\lambda}|e_n\rangle\langle e_n|$, for $\lambda \in\sigma(f)$. We will have $h(f)=\sum_n h(f_n)|e_n\rangle\langle e_n|=\sum_{\lambda\in\sigma(f)}h( \lambda) Q_\lambda$, for any function $h$ defined on the spectrum of $f$. 
\\
Notice that each projection $Q_\lambda$
 really is in $\M$, since $f$ is, while the projections $|e_n\rangle\langle e_n|$ do not necessarily belong to $\M$, in general; notice also that the rank of a projection $Q_\lambda$ has dimension equal to the trace of $Q_\lambda$. We have the following obvious relations
$$
\sum_{\lambda\in\sigma(f)}Q_\lambda=\unit,\qquad
\sum_{\lambda\in\sigma(f)}\tr{Q_\lambda}=N.
$$
Finally, we introduce the map 
$$
J:\{1,... N\}\longrightarrow \sigma(f), \qquad J(k)=\lambda \Leftrightarrow Q_\lambda e_k=e_k.
$$
\underline{Step 2}. We fix $f$ and its spectral representation. We want to define a family of maps $(K(t))_{t\ge 0}$ on ${\mathbb R}^N$ describing the action of the quantum semigroup on the sub-algebra of $\M$ generated by  the projections $(Q_\lambda)_{\lambda\in\sigma(f)}$. 
For any $t\ge 0$, we define the $N\times N$ real matrix $K(t)$ with elements
$$
(K(t))_{nm}=K_{nm}(t)=\frac{\tr{P_t(Q_{J(n)})Q_{J(m)}}}{\tr{Q_{J(n)}}\tr{Q_{J(m)}}}
\qquad\mbox{ for } n,m=1,...N.
$$
We shall identify the matrix $K(t)$ and the corresponding linear operator on ${\mathbb R}^N$, as is usual.
We highlight some properties of this matrix.
\\
({\bf a}) - The elements $(K(t))_{nm}$ are non negative.\\
Indeed, since the semigroup $P$ is positive, $P_t(Q_\lambda)$ is a positive operator, and so also $Q_\mu P_t(Q_\lambda)Q_\mu$ and its trace $\tr{P_t(Q_\lambda) Q_\mu}\ge 0$, for all $\lambda$ and $\mu$ in $\sigma(f)$.
\\
({\bf b}) - The matrix $K(t)$ is doubly stochastic, i.e. it is stochastic and
the uniform measure is invariant for it.
Indeed, $K(t)$ is symmetric by definition, due to the fact that $P$ is trace-symmetric and
\begin{eqnarray*}
\sum_{m=1}^N (K(t))_{nm}
&=& \sum_{m=1}^N \frac{\tr{P_t(Q_{J(n)})Q_{J(m)}}}{\tr{Q_{J(n)}}\tr{Q_{J(m)}}}\\
&=& \sum_{\lambda\in\sigma(f)}\;\sum_{m:J(m)=\lambda} \frac{\tr{P_t(Q_{J(n)})Q_{\lambda}}}{\tr{Q_{J(n)}}\tr{Q_{\lambda}}}\\
&=& \sum_{\lambda\in\sigma(f)} \frac{\tr{P_t(Q_{J(n)})Q_{\lambda}}}{\tr{Q_{J(n)}}}
=\frac{\tr{P_t(Q_{J(n)})}}{\tr{Q_{J(n)}}}=1,
\end{eqnarray*}
where the last equality is due to the fact that $P$ is trace preserving.
\\
({\bf c}) - For any function $h:\sigma(f)\sra {\mathbb R}$, we can associate two different objects
$$
\tilde h=(h(f_n))_{n=1,...N}\in{\mathbb R}^N,\qquad
M_h=\sum_{\lambda\in\sigma(f)}h(\lambda)Q_\lambda\in\M.
$$
If $g$ is another such function, and $\tilde g$ and $M_g$ are defined consequently, it is easy to verify that 
(here $\scal{\,}{\,}$ denotes the usual scalar product in ${\mathbb R}^N$)
\begin{eqnarray}\label{pluto}
\scal{K(t)\tilde h}{\tilde g}=\scal{\tilde h}{K(t)\tilde g}
&=&\sum_{n,m} h(f_n)g(f_m)\frac{\tr{P_t(Q_{J(n)})Q_{J(m)}}}{\tr{Q_{J(n)}}\tr{Q_{J(m)}}}\\
&=&\sum_{\lambda,\mu\in\sigma(f)}\, \sum_{m:J(m)=\lambda} \,\sum_{n:J(n)=\mu} h(\mu)g(\lambda)\frac{\tr{P_t(Q_{\mu})Q_{\lambda}}}{\tr{Q_{\mu}}\tr{Q_{\lambda}}}\nonumber\\
&=&\sum_{\lambda,\mu\in\sigma(f)}h(\mu)g(\lambda) \tr{P_t(Q_{\mu})Q_{\lambda}}\nonumber\\
&=&\tr{P_t(M_h)M_g}=\tr{P_t(M_g)M_h}.\nonumber
\end{eqnarray}
\underline{Step 3}.
Going back to the regularity conditions \eqref{reg-trace}, we shall use the previous relation with $h$ and $g$ power functions.
Consider that, for $h$ of the form $h(\lambda)=\lambda^p$, $p>0$, we shall have $\tilde h=(f_n^p)_n$ and $M_h=f^p$.
 We can then write
\begin{eqnarray}\label{pippo}
-\tr{f^{q/2}\LL(f^{q/2})}
&=&-\lim_{t\sra 0} \frac{1}{t}\tr{f^{q/2}(P_t(f^{q/2})-f^{q/2})}
\nonumber\\\mbox{(by \eqref{pluto})}
&=&-\lim_{t\sra 0} \frac{1}{t}\sum_{j,n}f_j^{q/2}K_{jn}(t)(f_n^{q/2}-f_j^{q/2})
\nonumber\\
\mbox{(since $K$ is symmetric)}
&=&\lim_{t\sra 0} \frac{1}{2t}\sum_{j,n}K_{jn}(t)(f_n^{q/2}-f_j^{q/2})^2.
\end{eqnarray}
Now we recall the usual elementary inequality, for $a$ and $b$ positive real numbers,\\
$$(a^{q/2}-b^{q/2})^2\le\frac{q^2}{4(q-1)}(a^{q-1}-b^{q-1})(a-b),$$
and, using this in \eqref{pippo}, we obtain 
\begin{eqnarray*}
-\tr{f^{q/2}\LL(f^{q/2})}&\le&\frac{q^2}{4(q-1)}\lim_{t\sra 0} \frac{1}{2t}\sum_{j,n}K_{jn}(t)(f_n^{q-1}-f_j^{q-1})(f_n-f_j)
\\
\mbox{(since $K$ is symmetric)}
&=&-\lim_{t\sra 0} \frac{1}{t}\sum_{j,n}f_j^{q-1}K_{jn}(t)(f_n-f_j)
\\\mbox{(by \eqref{pluto})}
&=&-\lim_{t\sra 0} \frac{1}{t}\tr{f^{q-1} (P_t(f)-f))}
\\
&=&-\tr{f^{q-1} \LL(f)},
\end{eqnarray*}
This proves (RC) and, by Proposition \ref{RC}, we deduce standard (WRC).
\end{proof}

\begin{remark}
In the case when the semigroup is trace preserving, but not symmetric, we can prove similar regularity properties, but with different constants. We can proceed as in the previous proof and use the elementary inequalities usually exploited for Markov chains (see, for instance, \cite{DSC})
$$
(b^{q/2}-a^{q/2})b^{q/2}\le\frac{q}{2}b^{q-1}(b-a), \qquad
\lg b^2-\lg a^2\ge \frac{2}{b}(b-a).
$$
We shall obtain something similar to (RC) and a (WRC) with constant $\beta=2$ 
$$
\EE_2(I_{2,q}f,I_{2,q}f)
\le
\frac{q}{2}\EE_q(f,f), \qquad 
2 \EE(I_{2,1}f,I_{2,1}f)
\le
\EE(\lg(\rho^{1/2} f \rho^{1/2})-\lg\rho,f).$$
\end{remark}

\subsection{Relations between the different variational inequalities}

In this subsection, we shall prove that LSI implies SGI and, with the help of (WRC), also MLSI. 

\begin{theorem}\label{LS-SG}
LSI$(c)$ implies SGI$(c^{-1})$.
\end{theorem}
This is exactly the generalization of what happens for commutative algebras. The weaker inequality gap $\ge c^{-1}/2$ was already proven in \cite{OZ} (Theorem 4.1). The result is now surely optimal since examples are known with spectral gap equal to the inverse of the log-Sobolev constant. In this note, we have concentrated on the finite dimensional case, but this proof works for any von Neumann algebra acting on a separable Hilbert space.
\\

\begin{proof}
We consider a real number $\varepsilon$ and $f\in\M$ such that $1+\e f> 0$ (so $f$ is selfadjoint)
and $\langle 1, f\rangle=\tr{\rho f}=0$. We shall denote $\bar f =\rho^{1/4}f\rho^{1/4}$. The idea is to proceed as for the analogous proof for commutative functions: consider the log-Sobolev inequality for $1+\e f$, divide by $\e^2$ and take the limit for $\e\sra 0$.
\\
By definition, a logarithmic Sobolev inequality with constant $c$ is written
$$
H(f)\le c\EE(f,f), \qquad
\mbox{ with }
H(f):=\langle f,T_2(f)\rangle-\|f\|_2^2\lg\|f\|_2.
$$
Writing computations in detail, when the argument is $1+\varepsilon f$, we have
$$
\|1+\e f\|_2^2=1+2\e \langle 1, f\rangle +\e^2 \tr{\bar f^2}=1+\e^2 \tr{\bar f^2}
$$
and
\begin{eqnarray}\label{A}
H(1+\e f)
&=& \tr{(\rho^{1/2}+\e {\bar f})^2\lg(\rho^{1/2}+\e {\bar f})-\frac{1}{2}(\rho^{1/2}+\e {\bar f})^2\lg\rho}
\nonumber\\
&&-\|1+\e f\|_2^2\lg\|1+\e f\|_2\nonumber\\
&=&\tr{(\rho^{1/2}+\e {\bar f})^2(\lg(\rho^{1/2}+\e {\bar f})-\lg\rho^{1/2})}
\\
&&-\frac{1}{2}(1+\e^2\tr{{\bar f}^2})\lg(1+\e^2\tr{{\bar f}^2}).\nonumber
\end{eqnarray}
The usual integral representation of the logarithm gives us 
\begin{eqnarray*}
\lg(\rho^{1/2}+\e {\bar f})&-&\lg\rho^{1/2}
= \int_0^{+\infty} \frac{1}{s+\rho^{1/2}}-\frac{1}{s+\rho^{1/2}+\e \bar f}\, ds\\
&=&\e  \int_0^{+\infty} \frac{1}{s+\rho^{1/2}}\bar f \frac{1}{s+\rho^{1/2}+\e \bar f}\, ds\\
&=&\e  \int_0^{+\infty} \frac{1}{s+\rho^{1/2}}\bar f \frac{1}{s+\rho^{1/2}}\, ds\\
&&-\e^2 \int_0^{+\infty} \frac{1}{s+\rho^{1/2}}\bar f \frac{1}{s+\rho^{1/2}}\bar f \frac{1}{s+\rho^{1/2}}\, ds\\
&&+\e^3 \int_0^{+\infty} \frac{1}{s+\rho^{1/2}}\bar f \frac{1}{s+\rho^{1/2}}\bar f \frac{1}{s+\rho^{1/2}}\bar f \frac{1}{s+\rho^{1/2}+\e \bar f}\, ds.
\end{eqnarray*}
Now we use the latter relation in \eqref{A}. Considering that we can also write 
$(\rho^{1/2}+\e {\bar f})^2=\rho+\e(\rho^{1/2}{\bar f}+\bar f \rho^{1/2})+\e^2\bar f^2$, we deduce
\begin{eqnarray}\label{C}
H(1+\e f)
&=& \e A(f)+\e^2 (B(f)-\frac{1}{2}\tr{{\bar f}^2})+\e^3\tr{...},
\end{eqnarray}
where
\begin{eqnarray*}
A(f)&=& \tr{\rho  \int_0^{+\infty} \frac{1}{s+\rho^{1/2}}\bar f \frac{1}{s+\rho^{1/2}}\, ds},\\
B(f)&=&\tr{(\rho^{1/2}{\bar f}+\bar f \rho^{1/2})\int_0^{+\infty} \frac{1}{s+\rho^{1/2}}\bar f \frac{1}{s+\rho^{1/2}}\, ds}\\
&&-\tr{\rho\int_0^{+\infty} \frac{1}{s+\rho^{1/2}}\bar f \frac{1}{s+\rho^{1/2}}\bar f \frac{1}{s+\rho^{1/2}}\, ds}.
\end{eqnarray*}
One can easily verify that the term $\e^3$ is multiplied by a bounded factor, so we can ignore it.
The first order term is equal to
\begin{eqnarray*}
A(f)=\int_0^{+\infty}\tr{\frac{\rho}{(s+\rho^{1/2})^2}\bar f}\, ds
&=& \tr{\int_0^{+\infty}\frac{\rho}{(s+\rho^{1/2})^2}\, ds\bar f}=\tr{\rho f}=0.
\end{eqnarray*}
So the central point is the estimate of the second order coefficient, which can be rewritten as
\begin{eqnarray}\label{B(f)}
B(f)&=&2\int_0^\infty \tr{\bar f \frac{1}{s+\rho^{1/2}}\bar f \frac{\rho^{1/2}}{s+\rho^{1/2}}}\, ds
-\int_0^\infty \tr{\bar f \frac{1}{s+\rho^{1/2}}\bar f \frac{\rho}{(s+\rho^{1/2})^2}}\, ds\nonumber\\
&=& \int_0^\infty \tr{\bar f \frac{1}{s+\rho^{1/2}}\bar f \frac{\rho^{1/2}}{s+\rho^{1/2}}\left(2-\frac{\rho^{1/2}}{s+\rho^{1/2}}\right)}\, ds\\
&=& \int_0^\infty \tr{ f \frac{\rho^{1/2}}{s+\rho^{1/2}} f \frac{2s+\rho^{1/2}}{(s+\rho^{1/2})^2}\rho}\, ds.\nonumber
\end{eqnarray}
We can use the spectral decomposition of the state $\rho$, $\rho=\sum_k\rho_kQ_k$, for some $(\rho_k)_k$ collection of strictly positive numbers and $(Q_n)_n$ orthogonal projections such that $\sum_nQ_n=\unit$. Writing the operator $f$ in block matrix form in this basis, $f=\sum_{k,j} Q_k fQ_j $, we obtain
\begin{equation}\label{B}
B(f) = \sum_{k,j} \rho_j^{1/2}\rho_k \tr{Q_k fQ_jf} \int_0^\infty \frac{2s+\rho_k^{1/2}}{(s+\rho_j^{1/2})(s+\rho_k^{1/2})^2}\, ds.
\end{equation}
So now we have to compute the integrals of real functions. For $a,b$ real positive numbers, $a\neq b$, we have
\begin{eqnarray*}
\int_0^\infty \frac{2s+a}{(s+a)^3}\, ds
&=&\left(\frac{-2}{s+a}+\frac{a}{2(s+a)^2}\right|_0^\infty=\frac{3}{2a}\,,
\\
\int_0^\infty \frac{2s+a}{(s+b)(s+a)^2}\, ds
&=&\int_0^\infty \frac{a-2b}{(b-a)^2}\left(\frac{1}{s+b}-\frac{1}{s+a}\right)
+\frac{a}{a-b}\frac{1}{(s+a)^2}\, ds\\
&=&\left(\frac{a-2b}{(b-a)^2}\lg\frac{s+b}{s+a}-\frac{a}{a-b}\frac{1}{s+a}\right|_0^\infty\\
&=&\frac{a-2b}{(b-a)^2}\lg\frac{a}{b}+\frac{1}{a-b}\,.
\end{eqnarray*}
Going back to $B(f)$, expression \eqref{B}, we shall use the first of the previous integral when $\rho_j=\rho_k$, with $a=\rho_k^{1/2}$, and the second otherwise, with $a=\rho_k^{1/2}$ and $b=\rho_j^{1/2}$. We obtain
\begin{eqnarray*}
B(f)&=&\frac{3}{2}\sum_{k,j:\,\rho_k=\rho_j} \rho_k \tr{Q_k fQ_jf}
\\&&+\sum_{k,j:\,\rho_k\neq \rho_j} \tr{Q_k fQ_jf}
\frac{\rho_k \rho_j^{1/2}}{\rho_k^{1/2}- \rho_j^{1/2}}
\left(1+\frac{\rho_k^{1/2} - 2\rho_j^{1/2}}{2(\rho_k^{1/2}- \rho_j^{1/2})}\,\lg\frac{\rho_k}{\rho_j}\right)
\end{eqnarray*}
where we can rewrite the second term summing on $k<j$, since $\tr{Q_k fQ_jf}=\tr{Q_j fQ_kf}$, and so 
\begin{eqnarray*}
B(f)&=&\frac{3}{2}\sum_{k,j:\:\,\rho_k=\rho_j} \rho_k \tr{Q_k fQ_jf}
\\&&+\sum_{k<j:\,\rho_k\neq \rho_j} \tr{Q_k fQ_jf}
(\rho_k \rho_j)^{1/2}
\left(1+\frac{\rho_k^{1/2}+\rho_j^{1/2}}{2(\rho_k^{1/2}- \rho_j^{1/2})}\,\lg\frac{\rho_k}{\rho_j}\right).
\end{eqnarray*}
Now notice that, for $t>0$, $\frac{\lg t}{t-1}\ge \frac{2}{t+1}$ (this inequality can be easily verified computing the derivatives of the two functions); then take $t=(\rho_k/\rho_j)^{1/2}$ and deduce
$$
\frac{\rho_k^{1/2}+\rho_j^{1/2}}{2(\rho_k^{1/2}- \rho_j^{1/2})}\,\lg\frac{\rho_k}{\rho_j}\ge 2,
$$
so that
$$
B(f)\ge \frac{3}{2}\sum_{k,j} \rho_k^{1/2} \rho_j^{1/2} \tr{Q_k fQ_jf}=\frac{3}{2}\tr{\bar f^2}.
$$
Summing up, using the lower bound for $B(f)$ in relation \eqref{C}, we have that there exists the limit
$$
\lim_{\e\sra 0}\frac{1}{\e^2}H(1+\e f)\ge \tr{\bar f^2}={\rm Var}(f).
$$
Moreover the LSI$(c)$ will give
$$
 {\rm Var}(f)\le \lim_{\e\sra 0}\frac{1}{\e^2}H(1+\e f)
 \le\lim_{\e\sra 0}\frac{c}{\e^2}\EE(1+\e f,1+\e f)=c \EE(f,f),
$$
which is a spectral gap inequality of constant $c^{-1}$. 
\end{proof}

\begin{remark}
For the reader interested in technical details, we should underline that the general idea of this proof is the same as the one followed in classical contexts and really also in \cite{OZ}, but we can improve the latter essentially since, in \cite{OZ}, the term $\left(2-\frac{\rho^{1/2}}{s+\rho^{1/2}}\right)$ in expression \eqref{B(f)} was dropped simply considering that it is not less than $1$, while here we can get a more precise lower bound.
\end{remark}

For the MLSI, we prove, in the following proposition, that it is implied by LSI when some (WRC) holds. Remember that, for classical Markov processes, (WRC) is always verified, in the standard form when the invariant law is also reversible.

\begin{proposition}\label{LS-MLS}
LSI$(c)$ and (WRC-$\beta$) imply MLSI$(\beta/(2c))$.
\end{proposition}

\begin{proof}
The basic ideas of this proof are the same as in classical known results. The mathematical objects have a slightly more complicated form here, but the technical difficulties disappear thanks to the fact that we have introduced the entropy and the log-Sobolev inequalities in a proper way, so that some mutual relations are maintained. 
\\
LSI can be written 
$$
H(f)=\langle f,T_2(f)\rangle-\|f\|_2^2\lg\|f\|_2\le c\EE(f,f)
$$
and we remember that
$$
E(f)= 2\langle I_{2,1}(f), T_2(I_{2,1}(f)) \rangle
-2 \| I_{2,1}(f) \|_2^2 \lg\| I_{2,1}(f) \|_2
=2H(I_{2,1}(f)).
$$
So we obtain, using first LSI and second the (WRC) with constant $\beta$, 
$$
E(f)=2H(I_{2,1}(f))\le 2c\EE(I_{2,1}(f),I_{2,1}(f))
\le2c \beta^{-1} \EE(\lg(\rho^{1/2} f \rho^{1/2})-\lg\rho,f),
$$
which is a MLSI with constant $\beta/(2c)$.
\end{proof}

{\bf About remaining open questions}.
\\
Really there can be a large amount of problems starting from these topics. However, concentrating only on the non-commutative version of the scheme written in Figure \ref{fig:commutative_case},
we can highlight two main missing points.
\\
1. Some of the implications (and relations with hypercontractivity) depend on regularity conditions, different but very similar. We think that the main problem about this aspect remains open: are the regularity conditions verified for all symmetric positive quantum semigroups? (as happens for the classical case) or can we at least find some general conditions under which they hold? We have some examples when they are verified (the trace case illustrated in Theorem \ref{trace}, and also a class of positive semigroups on $2\time 2$ matrices, see \cite{C}), but we know no counterexample.
\\
2. We have not been able by now to prove, whether MLSI implies SGI in general. We have some computations (see \cite{Preprint}) clarifying what happens for diagonal preserving semigroups on $M_2({\mathbb C})$: the implication holds, but what is maybe more interesting is that one can observe that the ideas used to prove the same result for the commutative case cannot work.
\medskip

{\bf Acknowledgements.} The financial support of MIUR FIRB 2010 project RBFR10COAQ
``Quantum Markov Semigroups and their Empirical Estimation'' and of PRIN project ``Problemi differenziali di evoluzione: approcci deterministici e stocastici e loro interazioni''
are gratefully acknowledged.


\begin{thebibliography}{999999999}

\bibitem{Vari}
C. An\'e; S. Blach\`ere; D. Chafai; P. Foug\`eres; I. Gentil; F. Malrieu; C. Roberto and G. Scheffer, Sur les in\'egalit\'es de Sobolev logarithmiques.,
vol. 10 Panoramas et Synth\`eses. S.M.F., Paris, 2000.

\bibitem{Ba}
 D. Bakry,
 L'hypercontractivit\'e et son utilisation en th\'eorie des semigroupes. Lectures on probability theory (Saint-Flour, 1992), 1Ð114, Lecture Notes in Math., 1581, Springer, Berlin, 1994. 
 \bibitem{BE}
 D. Bakry, M. \'Emery,  Diffusions hypercontractives. S\'eminaire de probabilit\'es, XIX, 1983/84, 177Ð206, Lecture Notes in Math., 1123, Springer, Berlin, 1985.

\bibitem{Bia} Ph. Biane,  {\it Free hypercontractivity}. Comm. Math. Phys. {\bf 184} (1997), no. 2, 457--474.

\bibitem{BT} S. Bobkov, P. Tetali, Modified logarithmic Sobolev inequalities in discrete settings. J. Theoret. Probab. 19 (2006), no. 2, 289Ð336.

\bibitem{BZ} T.~Bodineau, B.~Zegarlinski,
{\it Hypercontractivity via Spectral Theory},  Infinite Dimensional Analysis, Quantum Probability and Relat. Topics  {\bf 3} (2000), no.~1, 15--31.

\bibitem{Boz}  M. Bo\.zejko, {\it Ultracontractivity and strong Sobolev inequality
for $q$-Ornstein-Uhlenbeck semigroup $(-1<q<1)$}.
 Infin. Dimens. Anal. Quantum Probab. Relat. Top. {\bf 2} (1999), no. 2,
203--220.

\bibitem{Preprint}
R. Carbone, Modified log-Sobolev inequalities for Markov semigroups on $M_2({\mathbb C})$, Preprint.

\bibitem{C1} R. Carbone,
Exponential L2-convergence of some quantum Markov semigroups related to birth-and-death processes. Stochastic analysis and mathematical physics (Santiago, 1998), 1Ð22, Trends Math., BirkhŠuser Boston, Boston, MA, 2000. 

\bibitem{C} R. Carbone, {\it Optimal log-Sobolev inequality and hypercontractivity
for positive semigroups on $M\sb 2(\mathbb C)$}.  Infin. Dimens. Anal. Quantum
Probab. Relat. Top. {\bf 7} (2004), no. 3, 317--335.

\bibitem{CF}
R.~Carbone, F.~Fagnola, 
{\it Exponential $L^2$-convergence of quantum Markov semigroups on ${\mathcal{B}}(h)$},  Math. Notes \textbf{68}, no. 3-4,  (2000), 452--463.

\bibitem{C2}
R. Carbone, F. Fagnola, J.C. Garc\'ia, R. Quezada, Spectral properties of the two-photon absorption and emission process. J. Math. Phys. 49 (2008), no. 3, 18 pp.

\bibitem{CaSa} R. Carbone, E. Sasso, Hypercontractivity for a quantum Ornstein-Uhlenbeck semigroup. Probab. Theory Related Fields 140 (2008), no. 3-4, 505--522.

\bibitem{CaLi}
E.A.~Carlen, E.H~Lieb, {\it Optimal hypercontractivity for Fermi fields and related noncommutative integration inequalities}, Comm. Math. Phys. {\bf 155}, no. 1, (1993) 27--46.

\bibitem{ChSh} G. Chen, Y. Sheu, On the log-Sobolev constant for the simple random walk on the n-cycle: the even cases. J. Funct. Anal. 202 (2003), no. 2, 473--485.

\bibitem{DPP} P. Dai Pra, A. Paganoni, G. Posta, 
Entropy inequalities for unbounded spin systems. 
Ann. Probab. 30 (2002), no. 4, 1959--1976.


\bibitem{DSC} P.~Diaconis, L.~Saloff-Coste.
{\it Logarithmic Sobolev Inequalities for Finite Markov Chains},
Ann. Appl. Probab., no.3, \textbf{6}, (1996), 695--750.
	
\bibitem{Federbush}
P. Federbush, A partially alternate derivation of a proof of Nelson, J. Math. Phys. 10, 50Ð52 (1969). 

\bibitem{Gross} L.~Gross, {\it Logarithmic Sobolev inequalities}, {Amer. J. Math.} {\bf 97} (1975), 1061--1083.

\bibitem{Gross2}
L.~Gross, {\it Hypercontractivity and logarithmic Sobolev inequalities for the Clifford Dirichlet form}, Duke Math. J. {\bf 42} (1975), no. 3, 383--396. 

\bibitem{GZ} A. Guionnet, B. Zegarlinski, 
Lectures on logarithmic Sobolev inequalities. S\'eminaire de Probabilit\'es, XXXVI, 1Ð134, Lecture Notes in Math., 1801, Springer, Berlin, 2003.

\bibitem{KT}  M. J. Kastoryano, K. Temme, {\it Quantum logarithmic Sobolev inequalities and rapid mixing.} J. Math. Phys. 54 (2013), no. 5, 052202, 30 pp.

\bibitem{Krolak} I.~Kr\.olak, {\it Contractivity properties of Ornstein-Uhlenbeck semigroup for general commutation relations}. Math. Z. 250 (2005), no. 4, 915--937. 
 
\bibitem{Lindsay}
 J.M.~Lindsay, {\it Gaussian hypercontractivity revisited}. J. Funct. Anal. {\bf 92} (1990), no. 2, 313--324. 
 
\bibitem{Nelson}
E. Nelson, The free Markov field, J. Funct. Anal. 12, 211--227 (1973).

\bibitem{Saloff-Coste} L.~Saloff-Coste,
{\it Lectures on finite Markov chains}, 
Lectures on probability theory and statistics (Saint-Flour, 1996), 301--413, 
Lect. Notes Math. 1665, 
Springer, 1997. 

\bibitem{OZ}
R.~Olkiewicz and B.~Zegarlinski,  Hypercontractivity in noncommutative
$L_p$ Spaces, J. Funct. Anal., \textbf{161}, no.1 (1999), 246-285.


\end{thebibliography}
\end{document}